\documentclass[a4paper,11pt]{article}
\usepackage{stmaryrd}
\usepackage{bbm}
\usepackage{bm}
\usepackage[top=2.5cm, bottom=2.6cm, left=3cm, right=2.6cm]{geometry}
\usepackage{amsfonts}
\usepackage{mathrsfs}
\usepackage{amsmath}
\usepackage{latexsym}
\usepackage{color}
\usepackage{amscd}
\usepackage{amssymb}
\usepackage{amsthm}
\usepackage{latexsym}
\usepackage{indentfirst}

\usepackage{epic}
\usepackage{eepicemu}
\usepackage{graphicx}
\usepackage{epstopdf}

\def\u #1 #2{\mathcal U(#1, #2)}  %notation for ordinary tangles
\def\uhat #1 #2{\widehat{\mathcal U}(#1, #2)}  %affine tangles

\def\bmw #1{W_{#1}}  %stands for Birman-Wenzl
\def\w #1 #2{\bmw {#1}^{(#2)}}  %ideals in Birman-Wenzl
\def\V #1 #2{V_{#1}^{(#2)}}  %complement of W(n,r-2) in W(n, r)
\def\k #1 #2{ KT_{#1}^{(#2)}}
\def\hods{\unskip\kern.55em\ignorespaces}

\newtheorem{theorem}{Theorem}[section]

\newtheorem{lemma}{Lemma}[section]

\begin{document}

\title{Affine cellularity of affine Birman-Murakami-Wenzl algebras}
\author{Weideng Cui }
\date{}
\maketitle \abstract{We show that the affine BMW algebras are affine cellular algebras.}\\

 \large

\medskip
\section{Introduction}
The origin of the Birman-Murakami-Wenzl (BMW) algebras may be traced back to the important problem in knot theory of classifying knots (and links) up to isotopy. Shortly after the invention of the Jones link invariant in [J], Kauffman introduced a new invariant of regular isotopy for links in $S^{3},$ determined by certain skein relations [K]. Birman and Wenzl [BW] and independently Murakami [Mu] then defined a family of algebras, which were quotients of braid group algebras, and from which Kauffman's invariant could be recovered. These algebras, which were known as BMW algebras, were defined by generators and relations, but were implicitly modeled on certain algebras of tangles, whose definition was subsequently made explicitly by Morton and Traczyk [MT], as follows: let $S$ be a commutative unitary ring with invertible elements $\rho, q,$ and $\delta_{0}$ satisfying $\rho^{-1}-\rho=(q^{-1}-q)(\delta_{0}-1).$ The Kauffman tangle algebra $KT_{n, S}$ is the $S$-algebra of framed $(n, n)$-tangles in the disc cross the interval, modulo the following Kauffman skein relations:\vskip2mm

(1) Crossing relation:
 
(2) Untwisting relation:
 
(3) Free loop relation:  

\vskip2mm
Morton and Traczyk [MT] showed that the $n$-strand algebra $KT_{n, S}$ is free of rank $(2n-1)!!$ as a module over $S.$ In [MW], Morton and Wassermann have constructed an explicit basis of the BMW algebras and using it they have proved that the BMW algebras and the Kauffman tangle algebras are isomorphic. They are closely connected with the Artin braid group of type $A,$ Iwahori-Hecke algebras of type $A.$

In view of these relationships between the BMW algebras and objects of type $A,$ it is natural to generalize the BMW algebras for other types of Artin groups. Motivated by knot theory associated with the Artin group of type $B,$ H\"{a}ring-Oldenburg [HO] introduced the cyclotomic BMW algebras $W_{n, S,r}$ as a generalization of the BMW algebras such that the Ariki-Koike algebra $H_{n, S, r}$ (also known as the cyclotomic Hecke algebra of type $G(r,1,n)$) arises as a quotient of $W_{n, S,r},$ in the same way the Iwahori-Hecke algebra of type $A$ is a quotient of the BMW algebra. They are obtained from the BMW algebras by adding an extra generator $y_{1}$ satisfying a monic polynomial relation of finite order $r$ and imposing further relations modeled on type $B$ knot theory. When $r=1,$ one retrieves the original BMW algbras.

The cellularity of the BMW algebras has been shown by Xi [Xi] using the basis in [MW] (see also [E1]). It has been shown in [GM] that if $S$ is an integral domain and admissible in the sense of [WY1], then $W_{n, S,r}$ is a free $S$-module of rank $r^{n}(2n-1)!!,$ and is isomorphic to a cyclotomic version of the Kauffman tangle algebra. In [RX], Rui-Xu have also shown that $W_{n, S,r}$ is a free $S$-module under $\mathbf{u}$-admissible conditions on the parameters of $S.$ Goodman [G1] and Rui-Xu [RX] have obtained the weak cellularity results on $W_{n, S,r}$ over an admissible (a $\mathbf{u}$-admissible) ring respectively, Rui-Xu have obtained additional representation theoretic results, for example, classification of simple modules over a field. It is proved in [G2] that admissible conditions in [WY1] and $\mathbf{u}$-admissible condition in [RX] are equivalent for integral ground rings $S$ with $q-q^{-1}\neq 0.$ Yu [Yu] (see also [WY2-3]) has also shown that $W_{n, S,r}$ over an admissible ground ring $S$ is a free $S$-module of rank $r^{n}(2n-1)!!,$ is isomorphic to a cyclotomic version of the Kauffman tangle algebra, and is cellular in the sense of Graham and Lehrer.

It is also natural to ``affine" the BMW algebras to obtain BMW analogues of the extended affine Hecke algebras of type $A.$ On the one hand, H\"{a}ring-Oldenburg [HO] defined the affine BMW algebras algebraically by generators and relations; on the other hand, Goodman and Hauschild [GH] arrived at the definition of the affine Kauffman tangle algebras geometrically as algebras of framed $(n, n)$-tangles in the solid torus, modulo Kauffman skein relations. In [GH], They have shown that the two versions are isomorphic. Rui [Ru] has classifies the finite dimensional irreducible modules for affine BMW algebras over an algebraically closed field with arbitrary characteristic using the work of Goodman [G3].

Recently, Koenig and Xi in [KX2] has generalized this concept to algebras over a noetherian domain $k$ of not necessarily finite dimension, by introducing  the notion of an affine cellular algebra. The most important class of examples of affine cellular algebras, which has been discussed in [KX2], is given by the extended affine Hecke algebras of type $A.$ Recently, Guilhot and Miemietz have proved that affine Hecke algebras of rank two with generic parameters are affine cellular in [GuM]; Kleshchev and Loubert have proved that KLR algebras of finite type are affine cellular in [KL]. In [C1] and [C2], we have shown that the BLN-algebras which have been introduced by McGerty in [M], and the affine $q$-Schur algebras which have been defined by Lusztig in [L], are affine cellular algebras. In [C3], we have shown that the affine Brauer algebra, which has been introduced in [HO] (see also [GH]), is also affine cellular.

In this paper, we show that the techniques of [G1] can be modified to yield an affine cellular basis of the affine BMW algebras, thus we show that the affine BMW algebras are affine cellular algebras. Our proof relies heavily on the basis of the affine BMW algebras which has been constructed in [G1] and the fact that the extended affine Hecke algebras of type $A$ are affine cellular algebras, which has been shown in [KX2].

The organization of this paper is as follows. In Section 2, we will recall many notions, especially the affine BMW algebras $\widehat{W}_{n, S}$ and the extended affine Hecke algebras $\widehat{H}_{n, S}(q)$ of type $A,$ and define the affine cellular bases. In Section 3, we will recall a construction of some new basis of $\widehat{W}_{n, S}$ which has been given by Goodman. In Section 4, we prove our main result Theorem 4.1.

\section{Preliminaries}
\subsection{Definitions}
In the following, let $S$ be a commutative unitary ring containing elements $\rho, q$ and $\delta_{j}, j\geq 0$ with $\rho, q$ invertible, satisfying the relation $\rho^{-1}-\rho=(q^{-1}-q)(\delta_{0}-1).$\vskip2mm

\hspace*{-0.5cm}$\mathbf{Definition~2.1.}$ The affine Kauffman tangle algebra $\widehat{KT}_{n, S}$ is the $S$-algebra of framed $(n, n)$-tangles in the annulus cross the internal, modulo Kauffman skein relations, namely the crossing relation and untwisting relation, as given in the introduction, and the free loop relations: for $j\geq 0,$ $T\cup \Theta_{j}=\rho^{-j}\delta_{j}T,$ where $T\cup \Theta_{j}$ is the union of an affine tangle $T$ and a disjoint copy of the closed curve $\Theta_{j}$ that wraps $j$ times around the hole in the annulus cross the interval.

Affine tangles can be represented by affine tangle diagrams. These are pieces of link diagrams in the rectangle $\mathcal{R},$ with some number of endpoints of curves on the top and bottom boundaries of $\mathcal{R},$ and a distinguished vertical segment representing the hole in the annulus cross interval (we call this curve the flagpole). Affine tangle diagrams are regarded as equivalent if they are regularly isotopic; see [GM] for details. An affine $(n, n)$-tangle diagram is one with $n$ vertices (endpoints of curves) on the top, and $n$ vertices on the bottom edge of $\mathcal{R}.$ We label the vertices on the top edge from left to right as $\mathbf{1}, \mathbf{2},\ldots, \bm{n}$ and those on the bottom edge from left to right as $\overline{\mathbf{1}}, \overline{\mathbf{2}},\ldots, \overline{\bm{n}}.$ We order the vertices by $\mathbf{1}< \mathbf{2}<\cdots< \bm{n}<\overline{\bm{n}}<\cdots<\overline{\mathbf{2}}<\overline{\mathbf{1}}.$\vskip2mm

\hspace*{-0.5cm}$\mathbf{Definition~2.2.}$ The affine Birman-Murakami-Wenzl algebra $\widehat{W}_{n, S}$ is the $S$-algebra with generators $y_{1}^{\pm 1}, g_{i}^{\pm 1}$ and $e_{i}$ ($1\leq i\leq n-1$) and relations:\vskip2mm

(1) (Inverses) $g_{i}g_{i}^{-1}=g_{i}^{-1}g_{i}=1$ and $y_{1}y_{1}^{-1}=y_{1}^{-1}y_{1}=1.$

(2) (Idempotent relation) $e_{i}^{2}=\delta_{0} e_{i}.$

(3) (Type $B$ braid relations)

\hspace{0.7cm}(a) $g_{i}g_{i+1}g_{i}=g_{i+1}g_{i}g_{i+1}$ and $g_{i}g_{j}=g_{j}g_{i}$ if $|i-j|\geq 2.$

\hspace{0.7cm}(b) $y_{1}g_{1}y_{1}g_{1}=g_{1}y_{1}g_{1}y_{1}$ and $y_{1}g_{j}=g_{j}y_{1}$ if $j\geq 2.$

(4) (Commutation relations)

\hspace{0.7cm}(a) $g_{i}e_{j}=e_{j}g_{i}$ and $e_{i}e_{j}=e_{j}e_{i}$ if $|i-j|\geq 2.$

\hspace{0.7cm}(b) $y_{1}e_{j}=e_{j}y_{1}$ if $j\geq 2.$

(5) (Affine tangle relations)

\hspace{0.7cm}(a) $e_{i}e_{i\pm 1}e_{i}=e_{i}.$

\hspace{0.7cm}(b) $g_{i}g_{i\pm 1}e_{i}=e_{i\pm 1}e_{i}$ and $e_{i}g_{i\pm 1}g_{i}=e_{i}e_{i\pm 1}.$

\hspace{0.7cm}(c) For $j\geq 1,$ $e_{1}y_{1}^{j}e_{1}=\delta_{j}e_{1}.$

(6) (Kauffman skein relation) $g_{i}-g_{i}^{-1}=(q^{-1}-q)(e_{i}-1).$

(7) (Untwisting relations) $g_{i}e_{i}=e_{i}g_{i}=\rho^{-1}e_{i}$ and $e_{i}g_{i\pm 1}e_{i}=\rho e_{i}.$

(8) (Unwrapping relation) $e_{1}y_{1}g_{1}y_{1}=\rho e_{1}=y_{1}g_{1}y_{1}e_{1}.$\vskip2mm

Let $X_{1}, G_{i}, E_{i}$ denote the following affine tangle diagrams:

\begin{theorem} {\rm (see [GH])}
The affine BMW algebra $\widehat{W}_{n, S}$ is isomorphic to the affine Kauffman tangle algebra $\widehat{KT}_{n, S}$ by a map $\varphi$ determined by $\varphi(g_{i})=G_{i},$ $\varphi(e_{i})=E_{i}$ and $\varphi(y_{1})=\rho X_{1}.$
\end{theorem}

Because of Theorem 2.1, we will no longer take care to distinguish between affine BMW algebras and their realizations as algebras of tangles. We identify $e_{i}$ and $g_{i}$ with the corresponding affine tangle diagrams and $x_{1}=\rho^{-1} y_{1}$ with the affine tangle diagram $X_{1}.$ The ordinary BMW algebra $W_{n, S}$ embeds in the affine BMW algebra $\widehat{W}_{n, S}$ as the subalgebra generated by the $e_{i}$'s and $g_{i}$'s.

\subsection{The rank of tangle diagrams}

An ordinary or affine tangle diagram $T$ with $n$ strands is said to have rank $\leq r$ if it can be written as a product $T=T_{1}T_{2},$ where $T_{1}$ is an ordinary or affine $(r,n)$-tangle and $T_{2}$ is an ordinary or affine $(n, r)$-tangle.

\subsection{The algebra involution $*$ on BMW algebras}
Each of the ordinary and affine BMW algebras admits a unique anti-automorphism involution $*$, which is given by $a\mapsto a^{*},$ and fixes each of the generators $g_{i}, e_{i}$ (and $x_{1}$ in the affine case). For an ordinary or affine tangle diagram $T$ representing an element of one of these algebras, $T^{*}$ is the diagram obtained by flipping $T$ around a horizontal axis.

\subsection{The extended affine Hecke algebras of type $A$}

\hspace*{-0.5cm}$\mathbf{Definition~2.3.}$ Let $S$ be a commutative unitary ring with an invertible element $q.$ The extended affine Hecke algebra $\widehat{H}_{n, S}(q)$ of type $A$ over $S$ is the $S$-algebra with generators $t_{1}, \tau_{1}, \ldots, \tau_{n-1},$ with the following relations:\vskip2mm
(1) The generators $\tau_{i}$ are invertible, satisfy the braid relations, and $\tau_{i}-\tau_{i}^{-1}=q-q^{-1}.$

(2) The generators $t_{1}$ is invertible, $t_{1}\tau_{1}t_{1}\tau_{1}=\tau_{1}t_{1}\tau_{1}t_{1}$ and $t_{1}$ commutes with $\tau_{j}$ for $j\geq 2.$\vskip2mm

If a permutation $\pi\in \mathfrak{S}_{n}$ has a reduced expression $\pi=s_{i_1}s_{i_2}\cdots s_{i_{r}}$, then $\tau_{\pi}=\tau_{i_{1}}\tau_{i_{2}}\cdots \tau_{i_{r}},$ which is independent of the choice of the reduced expressions of $\pi.$ Define elements $t_{j}$ ($1\leq j\leq n$) in the affine Hecke algebra by $t_{j}=\tau_{j-1}\tau_{j-2}\cdots \tau_{1}t_{1}\tau_{1}\cdots \tau_{j-2}\tau_{j-1}.$ It is well-known that the affine Hecke algebra $\widehat{H}_{n, S}(q)$ is a free $S$-module with basis the set of elements $\tau_{\pi}\bm{t}^{b},$ where $\pi\in \mathfrak{S}_{n}$ and $\bm{t}^{b}$ denotes a Laurent monomial in $t_{1}, t_{2},\ldots, t_{n}.$

Let $S$ be a commutative ring with appropriate parameters $\rho, q, \delta_{j}.$ There is an algebra homomorphism $p: \widehat{W}_{n, S}\rightarrow \widehat{H}_{n, S}(q)$ determined by $g_{i}\mapsto \tau_{i}, e_{i}\mapsto 0$ and $x_{1}\mapsto t_{1}.$ The kernel of $p$ is the ideal $I_{n}$ spanned by affine tangle diagrams with rank strictly less than $n.$

The extended affine Hecke algebra $\widehat{H}_{n, S}(q)$ has a unique involutive algebra anti-automorphism $*$ fixing the generators $\tau_{i}$ and $t_1.$ Then we have $\tau_{w}=\tau_{w^{-1}}$ for any $w\in W,$ where $W$ is the extended affine Weyl group of type $A.$ The quotient map $p$ respects the involutions, that is, $p(x^{*})=p(x)^{*}.$

We have a linear section $t: \widehat{H}_{n, S}(q)\rightarrow \widehat{W}_{n, S}$ of the map $p$ determined by $t(\tau_{\pi}\bm{t}^{b})=g_{\pi}\bm{x}^{b}.$ Moreover, $t(x^{*})\equiv t(x)^{*}$ mod $I_{n}$ and $t(x)t(y)\equiv t(xy)$ mod $I_{n}$ for any $x, y\in \widehat{H}_{n, S}(q).$

\subsection{Affine cellular bases}
We recall the definition of affine cellularity from [KX2]. The version of the definition given here is slightly weaker than the original definition in [KX2].

Let $k$ be a noetherian domain. For a $k$-algebra $A$, a $k$-linear anti-automorphism $i$ of $A$ satisfying $i^{2}=id_{A}$ will be called a $k$-involution on $A.$ For two $k$-modules $V$ and $W,$ we denote by $\tau$ the map $V\otimes W\rightarrow W\otimes V$ given by $\tau(v\otimes w)=w\otimes v.$ If $B=k[x_{1},\ldots,x_{t}]/I$ for some ideal $I$ in a polynomial ring in finitely many variables $x_{1},\ldots, x_{t}$ over $k$, then $B$ is called an affine $k$-algebra.\vskip2mm

\hspace*{-0.5cm}$\mathbf{Definition~2.4.}$ Let $k$ be a noetherian domain, and let $A$ be a unitary $k$-algebra. An affine cell datum for $A$ consists of a $k$-involution $*$ on $A;$ a finite partially ordered set ($\Lambda, \geq$) and for each $\lambda\in \Lambda$ a finite set $\mathcal{T}(\lambda),$ an affine $k$-algebra $B_{\lambda}$ with a $k$-involution $\sigma_{\lambda};$ and a subset $\mathcal{C}=\{c_{s, t}^{\lambda}|~\lambda\in \Lambda ~\mathrm{and}~ s, t\in \mathcal{T}(\lambda)\}\subset A;$ with the following properties:\vskip2mm

(a) For each $\mu\in \Lambda,$ let $\widehat{A}^{\mu}$ be the right $B_{\mu}$-span of the $c_{s, t}^{\mu}$ for $s, t\in \mathcal{T}(\mu),$ then the set $\{c_{s, t}^{\mu}\}_{s, t\in \mathcal{T}(\mu)}$ is a $B_{\mu}$-basis of the right $B_{\mu}$-module $\widehat{A}^{\mu}$. Moreover, we have $A=\bigoplus_{\mu\in \Lambda} \widehat{A}^{\mu}$ (direct sums of $k$-modules).

(b) For each $\lambda\in \Lambda,$ let $\widehat{A}^{>\lambda}=\sum_{\mu>\lambda} \widehat{A}^{\mu}.$ Given $\lambda\in \Lambda,$ $s\in \mathcal{T}(\lambda),$ and $a\in A, b\in B_{\lambda},$ there exist coefficients $r_{v}^{s}(a)\in B_{\lambda}$ such that for all $t\in \mathcal{T}(\lambda),$ we have $$a\cdot (c_{s, t}^{\lambda}\cdot b)\equiv \sum_{v\in \mathcal{T}(\lambda)}c_{v, t}^{\lambda}\cdot r_{v}^{s}(a)b~~~(\mathrm{mod}~\widehat{A}^{>\lambda}),$$
where the coefficients $r_{v}^{s}(a)\in B_{\lambda}$ are independent of $t.$

(c) For all $\lambda\in \Lambda$ and $s, t\in \mathcal{T}(\lambda),$ and for any $b\in B_{\lambda},$ then we have $(c_{s, t}^{\lambda}\cdot b)^{*}\equiv c_{t, s}^{\lambda}\cdot\sigma_{\lambda}(b)$ mod $\widehat{A}^{>\lambda}.$\vskip2mm

$A$ is said to be an affine cellular algebra if it has an affine datum. For brevity, we will say that $\mathcal{C}$ is an affine cellular basis of $A.$

In [GL], they have introduced the notion of cellular algebras; in [KX1], Koenig and Xi have given an equivalent definition without using bases. In analogy to the proof of the equivalences of the two definitions, we can easily show that the definition of affine cellularity in Definition 2.4 is equivalent to the following basis-free definition of affine cellularity, which has been given by Koenig and Xi in [KX2].\vskip2mm

\hspace*{-0.5cm}$\mathbf{Definition~2.5.}$ (see [KX2, Definition 2.1]) Let $A$ be a unitary $k$-algebra with a $k$-involution $i$. A two-sided ideal $J$ in $A$ is called an affine cell ideal if and only if the following data are given and the following conditions are satisfied:\vskip2mm
$(1)$ We have $i(J)=J.$

$(2)$ There exist a free $k$-module of finite rank and an affine $k$-algebra $B$ with a $k$-involution $\sigma$ such that $\Delta :=V\otimes_{k} B$ is an $A$-$B$-bimodule, where the right $B$-module structure is induced by the right regular $B$-module $B_{B}$.

$(3)$ There is an $A$-$A$-bimodule isomorphism $\alpha :J\rightarrow \Delta\otimes_{B}\Delta',$ where $\Delta'=B\otimes_{k}V$ is a $B$-$A$-bimodule with the left $B$-module induced by the left regular $B$-module ${}_{B}B$ and with the right $A$-module structure defined by $(b\otimes v)a :=\tau(i(a)(v\otimes b))$ for $a\in A$, $b\in B$ and $v\in V$, such that the following diagram is commutative:

\[\begin{CD}
j   @>\alpha>>\Delta\otimes_{B}\Delta'\\
@ViVV                  @VVv\otimes b\otimes_{B}b'\otimes w\mapsto w\otimes \sigma(b')\otimes_{B}\sigma(b)\otimes vV\\
J         @>\alpha>>   \Delta\otimes_{B}\Delta'
\end{CD}\].

The algebra $A$ together with its $k$-involution $i$ is called affine cellular if and only if there is a chain of $i$-invariant two-sided ideals of $A$: $0=J_{0}\subset J_{1}\subset J_{2}\subset\cdots\subset J_{n}=A$, where each $J_{m}/J_{m-1}$ ($1\leq m\leq n$) is an affine cell ideal of $A/J_{m-1}$ (with respect to the involution induced by $i$ on the quotient).\vskip3mm

For an affine $k$-algebra $B$ with a $k$-involution $\sigma$, a free $k$-module $V$ of finite rank and a $k$-bilinear form $\varphi : V\otimes_{k} V\rightarrow B,$ denote by $\mathbb{A}(V,B,\varphi)$ the (possibly non-unital) algebra given as a $k$-module by $V\otimes_{k} B\otimes_{k}V,$ on which we impose the multiplication $(v_1\otimes b_1\otimes w_1)(v_2\otimes b_2\otimes w_2):$ $=v_1\otimes b_1\varphi(w_1,v_2)b_2\otimes w_2.$

For completeness, we will give the proof of the equivalences of the two definitions as follows.

\begin{proof}(1) Given Definition 2.4. Fix a maximal index $\lambda.$ Let $J(\lambda)$ be the right $B_{\lambda}$-span of the basis element $c_{s, t}^{\lambda}.$ By (b) and (c), this is a two-sided ideal and is fixed by the involution $*.$ Fix any index $t,$ define $V$ as the $k$-span of $c_{s, t}^{\lambda}$ (where $s$ varies). We define $\alpha: J(\lambda)\rightarrow \mathbb{A}(V,B_{\lambda},\varphi),$ where $\varphi$ is a $k$-bilinear form $\varphi :V\otimes_{k} V\rightarrow B_{\lambda},$ by sending $c_{u, v}^{\lambda}\cdot b$ to $c_{u, t}^{\lambda}\otimes b\otimes c_{v, t}^{\lambda}.$ It is easy to check that it is an isomorphism as an algebra and an $A$-$A$-bimodule. Thus $J(\lambda)$ is an affine cell ideal by [KX2, Proposition 2.2 and 2.3]. Continuing by induction, it follows that $A$ is affine cellular in the sense of Definition 2.5.

(2) Conversely, if an affine cell ideal $J$, in the sense of Definition 2.5, is given. We choose a set of right $B$-basis, say $\{c_{s}\}$, of $\Delta,$ and denote by $c_{s,t}\cdot b$ the inverse image under $\alpha$ of $c_{s}\otimes b\otimes_{B} 1 \otimes c_{t}.$ Since $\alpha$ is an $A$-$A$-bimodule isomorphism and $\Delta$ is an $A$-$B$-bimodule, (b) is satisfied. $(c_{s,t}\cdot b)^{*}=(c_{s,t})^{*}\cdot \sigma(b)$ follows from the required commutative diagram. This finishes the proof for those basis elements occurring in an affine cell ideal. Induction (on the length of the chain of ideals $J_{j}$) provides us with an affine cellular basis of the quotient algebra $A/J.$ Choose any preimages in $A$ of these basis elements together with a basis of $J$ as above we produce an affine cellular basis of $A.$\end{proof}

\hspace*{-0.5cm}$\mathbf{Remark ~2.1.}$ (1) The original definition in [KX2] also requires that $A$ has a $k$-module decomposition $A=J_{1}'\oplus J_{2}'\oplus\cdots J_{n}'$ (for some $n$) with $i(J_{l}')=J_{l}'$ for $1\leq l\leq n,$ which is equivalent to requiring that $(c_{s, t}^{\lambda}\cdot b)^{*}=c_{t, s}^{\lambda}\cdot\sigma_{\lambda}(b)$ for all $\lambda\in \Lambda$ and $s, t\in \mathcal{T}(\lambda),$ $b\in B_{\lambda}$ in Definition 2.4. However, we can check that the basic consequences in [KX2] remain valid with our weaker descriptions. For this reason, we have retained the terminology ``affine cellularity" for our weaker definition, rather than inventing some new terminology such as ``weak affine cellularity."

(2) In case $2\in k$ is invertible, one can easily check that our definition is equivalent to the original one.
\section{Some bases of the affine BMW algebras}

\subsection{$\mathbb{Z}$-Brauer diagrams}
We recall that a Brauer diagram is a tangle diagram in the plane, in which information about over- and under-crossings is ignored. Let $G$ be a group. A $G$-Brauer diagram (or $G$-connector) is a Brauer diagram in which each strand is endowed with an orientation and labeled by an element of the group $G.$ Two labelings are regarded as the same if the orientation of a strand is reversed and the group element associated to the strand is inverted.

Define a map $c$ from oriented affine $(n, n)$-tangle diagrams without closed loops to $\mathbb{Z}$-Brauer diagrams as follows. Let $a$ be an oriented affine $(n, n)$-tangle diagram without closed loops. If $s$ connects two vertices $\bm{v_{1}}$ to $\bm{v_{2}},$ include a curve $c(s)$ in $c(a)$ connecting the same vertices with the same orientation, and label the oriented strand $c(s)$ with the winding number of $s$ with respect to the flagpole.

The symmetric group $\mathfrak{S}_{n}$ can be regarded as the subset of $(n, n)$-Brauer diagrams consisting of diagrams with only vertical strands. $\mathfrak{S}_{n}$ acts on ordinary or $\mathbb{Z}$-labeled $(n, n)$-Brauer diagrams on the left and on the right by the usual multiplication of diagrams, that is, by stacking diagrams.

We consider a particular family of permutations in $\mathfrak{S}_{n}.$ Let $s$ be an integer, $0\leq s\leq n,$ with $s$ congruent to $n$ mod 2. Write $f=(n-s)/2.$ Following Enyang [E2], let $\mathcal{D}_{f, n}$ be the set of permutations $\pi\in \mathfrak{S}_{n}$ satisfying: \vskip2mm

(1) If $i, j$ are even numbers with $2\leq i< j\leq 2f,$ then $\pi(i)<\pi(j).$

(2) If $i$ is odd with $1\leq i\leq 2f-1,$ then $\pi(i)<\pi(i+1).$

(3) If $2f+1\leq i<j\leq n,$ then $\pi(i)<\pi(j).$\vskip2mm

Then $\mathcal{D}_{f, n}$ is a complete set of left coset representatives of $$\big((\mathbb{Z}_{2}\times \cdots \mathbb{Z}_{2})\rtimes \mathfrak{S}_{f}\big)\times \mathfrak{S}_{s}\subseteq \mathfrak{S}_{n},$$
where the $f$ copies of $\mathbb{Z}_{2}$ are generated by the transpositions $(2i-1, 2i)$ for $1\leq i\leq f;$ $\mathfrak{S}_{f}$ permutes the $f$ blocks [$2i-1, 2i$] among themselves; and $\mathfrak{S}_{s}$ acts on the last $s$ digits $\{2f+1,\ldots, n\}.$

An element $\pi$ of $\mathcal{D}_{f, n}$ factors as $\pi=\pi_{1}\pi_{2},$ where $\pi_{2}\in \mathcal{D}_{f, f},$ and $\pi_{1}$ is a $(2f, s)$ shuffle; i.e., $\pi_{1}$ preserves the order of $\{1,2,\ldots, 2f\}$ and of $\{2f+1,\ldots,2f+s=n\}.$ Moreover, we have $l(\pi)=l(\pi_{1})+l(\pi_{2}).$

For any $\mathbb{Z}$-Brauer diagram $D,$ let $D_{0}$ denote the underlying ordinary Brauer diagram; that is, $D_{0}$ is obtained from $D$ by forgetting the integer valued labels of the strands. If $D$ is a $\mathbb{Z}$-Brauer diagram with exactly $s$ vertical strands, then $D$ has a unique factorization$$D=\alpha d \beta^{-1},\eqno{(3.1)}$$
where $\alpha$ and $\beta$ are elements of $\mathcal{D}_{f, n},$ and $d$ has underlying Brauer diagram of the form $d_{0}=e_{1}e_{2}\cdots e_{2f-1}\pi,$ where $\pi$ is a permutation of $\{2f+1,\ldots, n\}.$

\subsection{Some bases}
We define the following commuting, but non-conjugate elements $$x_{j}=g_{j-1}\cdots g_{1}x_{1}g_{1}\cdots g_{j-1},~~~\mathrm{for}~1\leq j\leq n.$$

Given a $\mathbb{Z}$-Brauer diagram $D,$ we define an element $T_{D}$ whose associated $\mathbb{Z}$-Brauer diagram $c(T_{D})$ is equal to $D,$ as follows: suppose that $D$ has $2n$ vertices and $s$ vertical strands, and let $f=(n-s)/2.$ Let $D$ have the factorization $D=\alpha d \beta^{-1},$ where $\alpha, \beta\in \mathcal{D}_{f, n},$ and $d$ has underlying Brauer diagram of the form $d_{0}=e_{1}e_{3}\cdots e_{2f-1}\pi,$ with $\pi$ a permutation of $\{2f+1,\ldots, n\}.$

Define $$T_{d}=x_{1}^{a_{1}}\cdots x_{2f-1}^{a_{2f-1}}(e_{1}e_{3}\cdots e_{2f-1}g_{\pi})x_{1}^{c_{1}}x_{3}^{c_{3}}\cdots x_{2f-1}^{c_{2f-1}}x_{2f+1}^{b_{2f+1}}\cdots x_{n}^{b_{n}},\eqno{(3.2)}$$
where the exponents are determined as follows: if $d$ has a horizontal strand beginning at $\bm{i}$ with integer valued label $k,$ then $c_{i}=k;$ and $c_{i}=0$ otherwise. If $d$ has a vertical strand beginning at $\bm{i}$ with integer valued label $k,$ then $b_{i}=k;$ and $b_{i}=0$ otherwise. If $d$ has a horizontal strand ending at $\overline{\bm{i}}$ with integer valued label $k,$ then $a_{i}=k;$ and $a_{i}=0$ otherwise.

Finally, we set $T_{D}=g_{\alpha}T_{d}g_{\beta}^{*}.$ Then $c(T_{D})=D.$

Recall that $S$ is called weakly admissible if $e_{1}$ is not a torsion element. The following theorem gives a set of basis of $\widehat{W}_{n, S}$ involving monomials in the commuting elements $x_{j}.$

\begin{theorem} {\rm (see [G1, Theorem~3.17])}
Let $S$ be weakly admissible. The set $\mathbb{U}=\{T_{D}|~D~ \mathrm{is~ a }~ \mathbb{Z}\mathrm{-Brauer ~diagram}\}$ is an $S$-basis of $\widehat{W}_{n, S}.$
\end{theorem}

Let $D$ be a $\mathbb{Z}$-Brauer diagram, with $2n$ vertices and $s$ vertical strands, having factorization $D=\alpha d \beta^{-1},$ and let $T_{d}$ be defined as in (3.2) and let $T_{D}=g_{\alpha}T_{d}g_{\beta}^{*}.$ Factor $\alpha$ as $\alpha=\alpha_{1}\alpha_{2},$ with $\alpha_{1}$ a $(2f, s)$-shuffle, and $\alpha_{2}\in \mathcal{D}_{f, f},$ and factor $\beta$ similarly. Then we can rewrite $T_{D}$ as follows:$$T_{D}=g_{\alpha_{1}}\big[g_{\alpha_{2}}\bm{x}^{a}(e_{1}e_{3}\cdots e_{2f-1})\bm{x}^{c}(g_{\beta_{2}})^{*}\big]\big(g_{\pi}\bm{x}^{b}\big)(g_{\beta_{1}})^{*}, \eqno{(3.3)}$$
where $\bm{x}^{a}$ is short for $x_{1}^{a_{1}}\cdots x_{2f-1}^{a_{2f-1}},$ and similarly for $\bm{x}^{c},$ while $\bm{x}^{b}$ denotes $x_{2f+1}^{b_{2f+1}}\cdots x_{n}^{b_{n}}.$

\section{Affine cellular bases of affine BMW algebras}

\subsection{Tensor products of affine tangle diagrams}
The category of affine $(k, l)$-tangle diagrams is not a tensor category in any evident fashion. Nevertheless, we can define a tensor product of affine tangle diagrams, as follows. Let $T_{1}$ and $T_{2}$ be affine tangle diagrams (say of size $(a, a)$ and $(b, b)$ respectively), and suppose that $T_{2}$ has no closed loops. Then $T_{1}\odot T_{2}$ is obtained by replacing the flagpole in the affine tangle diagram $T_{2}$ with the entire affine tangle diagram $T_{1}.$ If we regard $T_{1}$ and $T_{2}$ as representing framed tangles in the annulus cross the interval $A\times I,$ then $T_{1}\odot T_{2}$ is obtained by inserting the entire copy of $A\times I$ containing $T_{1}$ into the hole of the copy of $A\times I$ containing $T_{2}.$

Then $T_{1}\otimes T_{2}\mapsto T_{1}\odot T_{2}$ determines a linear map from $\widehat{W}_{a, S}\otimes \widehat{W}_{b, S}$ into $\widehat{W}_{a+b, S}.$ Note that $(T_{1}\odot T_{2})^{*}=T_{1}^{*}\odot T_{2}^{*}.$

These maps of affine BMW algebras are not algebra homomorphism. In fact, we have $$(1\odot e_{1})(1\odot x_{1})(1\odot e_{1})=z\odot e_{1},$$
where $z$ is a (non-scalar) central element in $\widehat{W}_{a, S}.$ Nevertheless, we have $$(A\odot B)(S\odot T)=AS \odot BT,$$
if no closed loops are produced in the product $BT,$ in particular, if at least one of $B$ and $T$ has no horizontal strands.

\subsection{Affine cellular bases}

Using (3.3) and the remarks in Section 4.1, we can rewrite the elements $T_{D}$ in (3.3) in the following form:$$T_{D}=g_{\alpha_{1}}\big(\big[g_{\alpha_{2}}\bm{x}^{a}(e_{1}e_{3}\cdots e_{2f-1})\bm{x}^{c}(g_{\beta_{2}})^{*}\big]\odot \big(g_{\pi}\bm{x}^{b}\big)(g_{\beta_{1}})^{*}. \eqno{(4.1)}$$
Here, $D$ is a $\mathbb{Z}$-Brauer diagram with $s$ vertical strands and $f=(n-s)/2;$ $\alpha_{1}$ and $\beta_{1}$ are $(n-s, s)$-shuffles; $\pi\in \mathfrak{S}_{s}$ and $\bm{x}^{b}=x_{1}^{b_{1}}\cdots x_{s}^{b_{s}}.$ Moreover, $\alpha_{2}$ and $\beta_{2}$ are elements of $\mathcal{D}_{f, f},$ $\bm{x}^{a}=x_{1}^{a_{1}}x_{3}^{a_{3}}\cdots x_{2f-1}^{a_{2f-1}},$ and similarly for $\bm{x}^{c}.$

The affine $(2f, 2f)$-tangle diagram $$T=g_{\alpha_{2}}\bm{x}^{a}(e_{1}e_{3}\cdots e_{2f-1})\bm{x}^{c}(g_{\beta_{2}})^{*}$$
is stratified and flagpole descending in the sense of the definition given in [G1, Definition 3.5], with no vertical strands and no closed loops. Conversely, any stratified and flagpole descending affine $(2f, 2f)$-tangle diagram with no vertical strands and no closed loops is regularly isotopic to one of the form.

Note that we can factor $T$ as $T=xy^{*},$ where $x$ and $y$ are stratified and flagpole descending affine $(0, 2f)$-tangle diagrams with no closed loops, namely$$x=g_{\alpha_{2}}\bm{x}^{a}(\cap_{2f-1}\cdots \cap_{3}\cap_{1}),~~~\mathrm{and}~~~y=g_{\beta_{2}}\bm{x}^{c}(\cap_{2f-1}\cdots \cap_{3}\cap_{1}),$$
where $\cap_{i}$ is the lower half of $e_{i}.$

In fact, any stratified and flagpole descending affine $(0, 2f)$-tangle diagram with no closed loops is regularly isotopic to one of the form.

\begin{lemma} {\rm (see [G1, Lemma~4.1])}
The set of $T_{D}\in \mathbb{U}$ with $s$ vertical strands equals the set of elements $$g_{\alpha}(xy^{*}\odot g_{\pi}\bm{x}^{b})(g_{\beta})^{*},$$
where $x, y$ are stratified and flagpole descending affine $(0, n-s)$-tangle diagrams with no closed loops or self-crossings of strands; $\alpha$ and $\beta$ are $(n-s, s)$-shuffles; $\pi\in \mathfrak{S}_{s}$ and $\bm{x}^{b}=x_{1}^{b_{1}}\cdots x_{s}^{b_{s}}.$
\end{lemma}

For each $s$ with $s\leq n$ and $n-s$ even, let $V_{n}^{s}$ be the span in $\widehat{W}_{n, S}$ of the set of elements $T_{D}\in \mathbb{U}$ with $s$ vertical strands.
\begin{lemma}
For each $s,$ let $\mathbb{B}_{s}$ be a basis of $\widehat{H}_{n, S}(q).$ Let $\Sigma_{s}$ be the set of elements $$g_{\alpha}\big(xy^{*}\odot t(b)\big)(g_{\beta})^{*}\in \widehat{W}_{n, S},$$
such that $x, y$ are stratified and flagpole descending affine $(0, n-s)$-tangle diagrams with no closed loops or self-crossings of strands; $\alpha$ and $\beta$ are $(n-s, s)$-shuffles; and $b\in \mathbb{B}_{s}.$ Then $\Sigma_{s}$ is a basis of $V_{n}^{s}.$
\end{lemma}
\begin{proof} Recall that $\{\tau_{\pi}\bm{t}^{b}|~\pi\in \mathfrak{S}_{s}~\mathrm{and}~b_{i}\in \mathbb{Z}~\mathrm{for}~1\leq i\leq n\}$ is a basis of $\widehat{H}_{n, S}(q),$ and that $g_{\pi}\bm{x}^{b}=t(\tau_{\pi}\bm{t}^{b}).$ It follows from this and from Lemma 4.1 that $V_{n}^{s}$ is the direct sum over $(\alpha, \beta, x, y)$ of $$V_{n}^{s}(\alpha, \beta, x, y)=\big\{g_{\alpha}\big(xy^{*}\odot t(u)\big)(g_{\beta})^{*}|~u\in \widehat{H}_{n, S}(q)\big\}$$
and that $u\mapsto g_{\alpha}\big(T\odot t(u)\big)(g_{\beta})^{*}$ is injective. This implies the result.
\end{proof}

From [KX2, Theorem 5.7], we know that the extended affine Hecke algebra $\widehat{H}_{s, S}(q)$ of type $A$, for each $s,$ is affine cellular. For each $s$ with $s\leq n$ and $n-s$ even, Let ($*, \Lambda_{s}, \{\mathcal{T}(\lambda)\}_{\lambda\in \Lambda_{s}}, \{B_{s,\lambda}, \sigma_{s, \lambda}\}_{\lambda\in \Lambda_{s}}, \mathcal{C}_{s}=\{c_{s, t}^{\lambda}|~\lambda\in \Lambda_{s}, s,t\in \mathcal{T}(\lambda)\}$) be an affine cellular basis of the affine Hecke algebra $\widehat{H}_{s, S}(q),$ where $\Lambda_{s}$ is the set of all partitions of $s$ under the dominance order $\unlhd;$ the affine $S$-algebra $B_{s,\lambda}$ with a $k$-involution $\sigma_{s, \lambda},$ for each $\lambda\in \Lambda_{s},$ is the representation ring of a product of general linear groups $GL_{t}.$

Let $\Lambda=\{(s, \lambda)|~s\leq n ~\mathrm{and}~ n-s ~\mathrm{even}, \lambda\in \Lambda_{s}\}$ with a partial order $(s, \lambda)\geq (t, \mu)$ if $s< t$ or if $s=t$ and $\lambda\geq \mu$ (i.e. $\lambda\unrhd\mu$) in $\Lambda_{s}.$ For each pair $(s, \lambda)\in \Lambda,$ we take $\mathcal{T}(s, \lambda)$ to be the set of triples $(\alpha, x, u),$ where $\alpha$ is an $(n-s, s)$-shuffle; $x$ is a stratified and flagpole descending affine $(0, n-s)$-tangle diagrams with no closed loops or self-crossings of strands; and $u\in \mathcal{T}(\lambda).$ For any $b\in B_{s, \lambda},$ define $$c_{(\alpha,x,u),(\beta,y,v)}^{(s,\lambda)}\cdot b=g_{\alpha}\big(xy^{*}\odot t(c_{u,v}^{\lambda}\cdot b)\big)(g_{\beta})^{*},$$
and $\mathcal{C}$ to be the set of all $c_{(\alpha,x,u),(\beta,y,v)}^{(s,\lambda)}.$

\begin{lemma}

$(c_{(\alpha,x,u),(\beta,y,v)}^{(s,\lambda)}\cdot b)^{*}\equiv c_{(\beta,y,v), (\alpha,x,u)}^{(s,\lambda)}\cdot\sigma_{s, \lambda}(b)$ ~$\mathrm{mod}$~ $\widehat{W}_{n, S}^{>(s, \lambda)}.$
\end{lemma}
\begin{proof}
$\big(g_{\alpha}\big(xy^{*}\odot t(c_{u,v}^{\lambda}\cdot b)\big)(g_{\beta}\big)^{*})^{*}=g_{\beta}\big(yx^{*}\odot t(c_{u,v}^{\lambda}\cdot b)^{*}\big)(g_{\alpha})^{*},$ and $t(c_{u,v}^{\lambda}\cdot b)^{*}\equiv t((c_{u,v}^{\lambda}\cdot b)^{*})=t(c_{v,u}^{\lambda}\cdot \sigma_{s, \lambda}(b))$ modulo the span of diagrams of rank $< s.$ Hence $(c_{(\alpha,x,u),(\beta,y,v)}^{(s,\lambda)}\cdot b)^{*}\equiv c_{(\beta,y,v), (\alpha,x,u)}^{(s,\lambda)}\cdot\sigma_{s, \lambda}(b)$ modulo the span of diagrams of rank $< s.$
\end{proof}

\begin{lemma}
For any affine $(n-s, n-s)$-tangle diagram $A$ and affine $(s, s)$-tangle diagram $B,$ and for any $b\in B_{\lambda},$ $(A\odot B)(xy^{*}\odot t(c_{u,v}^{\lambda}\cdot b))$ can be written as a linear combination of elements $(x'y^{*}\odot t(c_{u,v}^{\lambda}\cdot b'b)),$ modulo $\widehat{W}_{n, S}^{>(s, \lambda)}$ with $b'$ independent of $y$ and $v.$
\end{lemma}
\begin{proof}
We have $(A\odot B)(xy^{*}\odot t(\tau_{\pi}\bm{t}^{b}))=Axy^{*}\odot Bt(\tau_{\pi}\bm{t}^{b}),$ because $t(\tau_{\pi}\bm{t}^{b})$ has only vertical strands. Therefore, also $(A\odot B)(xy^{*}\odot t(c_{u,v}^{\lambda}\cdot b))=Axy^{*}\odot Bt(c_{u,v}^{\lambda}\cdot b).$

Note that $Ax$ is an affine $(0, n-s)$-tangle, and can be reduced using the algorithm of the proof of Proposition 2.18 and 2.19 in [GH] to a linear combination of stratified, flagpole descending $(0, n-s)$-tangle $x'$ without closed loops. The process does not affect $y^{*}.$

If $B$ has rank strictly less than $s,$ then the product $(A\odot B)(xy^{*}\odot t(c_{u,v}^{\lambda}\cdot b))$ is a linear combination of basis elements $T_{D}$ with fewer than $s$ vertical strands, so belong to $\widehat{W}_{n, S}^{>(s, \lambda)}.$

Otherwise, we can suppose that $B=g_{\sigma}\bm{x}^{b}.$ Then $Bt(c_{u,v}^{\lambda}\cdot b)=t(\tau_{\sigma}\bm{t}^{b})t(c_{u,v}^{\lambda}\cdot b)\equiv t(\tau_{\sigma}\bm{t}^{b}c_{u,v}^{\lambda}\cdot b)$ modulo the span of diagrams with fewer than $s$ vertical strands. Moreover, $t(\tau_{\sigma}\bm{t}^{b}c_{u,v}^{\lambda}\cdot b)$ is a linear combination of elements $t(c_{u',v}^{\lambda}\cdot b'b),$ modulo $t(\widehat{H}_{s, S}^{>\lambda}),$ with $b'$ independent of $v,$ by the affine cellularity of the basis $\mathcal{C}_{s}$ of $\widehat{H}_{s, S}.$
\end{proof}

\begin{theorem}
The affine BMW algebra $\widehat{W}_{n, S}$ over a weakly admissible noetherian domain $S$ is an affine cellular algebra with an affine cellular basis $(*, \Lambda, \{\mathcal{T}(s, \lambda)\}_{(s,\lambda)\in \Lambda}, \{B_{s,\lambda}, \sigma_{s, \lambda}\}_{(s,\lambda)\in \Lambda}, \mathcal{C}=\{c_{(\alpha,x,u),(\beta,y,v)}^{(s,\lambda)}\}).$
\end{theorem}

\begin{proof}
For each $(s, \lambda)\in \Lambda,$ let $\widehat{W}^{(s,\lambda)}$ be the right $B_{\lambda}$-span of the basis element $c_{(\alpha,x,u),(\beta,y,v)}^{(s,\lambda)}$ for $(\alpha,x,u),(\beta,y,v)\in \mathcal{T}(s, \lambda).$ It follows from (the proof of) Lemma 4.2 that the set $\{c_{(\alpha,x,u),(\beta,y,v)}^{(s,\lambda)}\}$ is a right $B_{\lambda}$-basis of $\widehat{W}^{(s,\lambda)}.$ Theorem 3.1 and Lemma 4.2 implies that $\widehat{W}_{n, S}=\bigoplus_{(s,\lambda)\in \Lambda} \widehat{W}^{(s,\lambda)},$ and property (c) holds by Lemma 4.3. It remains to verify axiom (b) for affine cellular bases. Thus we have to show that for $w\in \widehat{W}_{n, S},$ and for an element $c_{(\alpha,x,u),(\beta,y,v)}^{(s,\lambda)}\cdot b=g_{\alpha}\big(xy^{*}\odot t(c_{u,v}^{\lambda}\cdot b)\big)(g_{\beta})^{*}$ $(b\in B_{\lambda})$, the product $$w\cdot g_{\alpha}\big(xy^{*}\odot t(c_{u,v}^{\lambda}\cdot b)\big)(g_{\beta})^{*}\eqno{(4.2)}$$
can be written as a linear combination of elements $$g_{\alpha'}\big(x'y^{*}\odot t(c_{u',v}^{\lambda}\cdot b'b)\big)(g_{\beta})^{*},$$
modulo $\widehat{W}_{n, S}^{>(s, \lambda)}$ with $b'$ independent of $(\beta, y, v).$

It suffices to consider products as in (4.2) with $w$ equal to $e_{i}$ or to $g_{i}$ for some $i,$ or to $w=x_{1}.$ Then the proof is similar to the proof of [G1, Theorem 4.5], the proof that he has given there can be easily adapted to give a proof of this theorem repeatedly using Lemma 4.4, we will omit the details.
\end{proof}

Let now $A$ be an affine cellular algebra with a cell chain $0=J_{0}\subset J_{1}\subset J_{2}\subset\cdots\subset J_{n}=A$, such that each subquotient $J_{i}/J_{i-1}$ is an affine cell ideal of $A/J_{i-1}$. Then $J_{i}/J_{i-1}$ is isomorphic to $\mathbb{A}(V_{i},B_{i},\varphi_{i})$ for some free $k$-module $V_{i}$ of finite rank, an affine $k$-algebra $B_{i}$ and a $k$-bilinear form $\varphi_{i} :V_{i}\otimes_{k} V_{i}\rightarrow B_{i}.$ Let $(\varphi_{st}^{i})$ be the matrix representing the bilinear form $\varphi_{i}$ with respect to some choices of basis of $V_{i}.$ Then Koenig and Xi obtain a parameterisation of simple modules over an affine cellular algebra by establishing a bijection between isomorphism classes of simple $A$-modules and the set $$\{(j, m)| 1\leq j\leq n, m\in \mathrm{MaxSpec(B_{j})~ such~ that~ some} ~\varphi_{st}^{i}\notin \mathrm{m}\},$$
where $\mathrm{MaxSpec(B_{j})}$ denotes the maximal ideal spectrum of $B_{j}.$ \vskip3mm

\hspace*{-0.5cm}$\mathbf{Remark ~4.1.}$ We conjecture that affine BMW algebras are affine cellular in the sense of Koenig and Xi. Presumably the methods in [Yu] can be adapted to give an affine cellular basis of affine BMW algebras in the sense of Koenig and Xi. The key point is to construct an appropriate affine cellular basis for the extended affine Hecke algebra of type $A,$ which is lifted by $t$ to a set of $\widehat{W}_{n, S}$ which is still compatible with the anti-involution $*$ on $\widehat{W}_{n, S}.$

\vskip3mm Acknowledgements. I would like to express my sincere gratitude to Professor Goodman for very helpful correspondence, for giving very useful comments on affine cellular bases, and for providing me with the figures in his paper [G1] and allowing me to use them. I would also like to express my sincere gratitude to Professor Hebing Rui for very helpful correspondence and comments on this paper, from which I have learnt a lot.

%/////////////////////////////////////////////////////////////////////////////////////////////////////////////////////////////////////////////////////////

%by (\ref{stn})
%\begin{align*}
%  t_{n+1}&\geq\frac{1}{t_n+(1-t_n)\frac{1}{c}}\\
%  1-t_{n+1}&\leq \frac{(1-t_n)(\frac{1}{c}-1)}{t_n+(1-t_n)\frac{1}{c}}=\frac{(1-t_n)(\frac{1}{c}-1)}{(1-t_n)(\frac{1}{c}-1)+1}\\
%  \frac{1}{1-t_{n+1}}&\geq \frac{1}{(1-t_n)(\frac{1}{c}-1)+1}
%\end{align*}
%Since $c\geq \frac{1}{2}$, $\frac{1}{c}-1\geq 1$, we have
%$$\frac{1}{1-t_{n+1}}\geq \frac{1}{1-t_n}+1\geq\cdots\geq\frac{1}{1-t_1}+n\geq n+1$$
%i.e. $1-t_{n+1}\leq \frac{1}{n+1}$
%$$0\leq v_{2n+1}-x^*\leq v_{2n+1}-v_{2n}\leq v_{2n+1}-t_nv_{2n+1}\geq \frac{1}{n}v_{2n+1}$$
%So $\|v_{2n+1}-x^*\|\leq \frac{N}{n}\|\bar{v}\|$, where $N$ is the normal constant of $P$.
%/////////////////////////////////////////////////////////////////////////////////////////////////////////////////////////////////

Mathematical Sciences Center, Tsinghua University, Jin Chun Yuan West Building

Beijing, 100084, P. R. China.

E-mail address: cwdeng@amss.ac.cn

\end{document}